\documentclass[11 pt,leqno]{amsart}
\theoremstyle{definition}
\usepackage{indentfirst, amssymb,amsmath,amsthm}
\usepackage{csquotes}
\usepackage{hyperref}
\newtheorem{ques}{Question}[section]
\newtheorem{theo}{Theorem}[section]
\newtheorem{lem}{Lemma}[section]
\newtheorem{cor}{Corollary}[section]

\newtheorem{exm}{Example}[section]
\newtheorem{defi}{Definition}[section]
\newtheorem{rem}{Remark}[section]

\newcommand{\ol}{\overline}
\newcommand{\be}{\begin{equation}}
\newcommand{\ee}{\end{equation}}
\newcommand{\beas}{\begin{eqnarray*}}
\newcommand{\eeas}{\end{eqnarray*}}
\newcommand{\bea}{\begin{eqnarray}}
\newcommand{\eea}{\end{eqnarray}}

\numberwithin{equation}{section}
\begin{document}
\title[Some results on the unique range sets]{Some results on the unique range sets}
\date{}
\author[B. Chakraborty, et al.]{Bikash Chakraborty$^{1}$, Jayanta Kamila$^{2}$, Amit Kumar Pal$^{3}$ and Sudip Saha$^{4}$}
\date{}
\address{$^{1}$Department of Mathematics, Ramakrishna Mission Vivekananda Centenary College, Rahara,
West Bengal 700 118, India.}
\email{bikashchakraborty.math@yahoo.com, bikash@rkmvccrahara.org}
\address{$^{2}$Department of Mathematics, Ramakrishna Mission Vivekananda Centenary College, Rahara,
West Bengal 700 118, India.}
\email{kamilajayanta@gmail.com}
\address{$^{3}$Department of Mathematics, University of Kalyani, Kalyani, West Bengal 741 235, India.}
\email{mail4amitpal@gmail.com}
\address{$^{4}$Department of Mathematics, Ramakrishna Mission Vivekananda Centenary College, Rahara,
West Bengal 700 118, India.}
\email{sudipsaha814@gmail.com}
\maketitle
\let\thefootnote\relax
\begin{abstract} In this paper, we exhibit the equivalence between different notions of unique range sets, namely, unique range sets, weighted unique range sets and weak-weighted unique range sets under certain conditions.\par Also, we present some uniqueness theorems which show how two meromorphic functions
are uniquely determined by their two finite shared sets. Moreover, in the last section, we make some observations that help us to construct other new classes of unique range sets.
\end{abstract}
\section{Introduction: Unique range sets}
We use $M(\mathbb{C})$ to denote the field of all meromorphic functions. Let $f\in M(\mathbb{C})$ and $S\subset\mathbb{C}\cup\{\infty\}$ be a non-empty set with distinct elements. We set $$E_{f}(S)=\bigcup\limits_{a\in S}\{z~:~f(z)-a=0\},$$
where a zero of $f-a$ with multiplicity $m$ counts $m$ times in $E_{f}(S)$. Let $\ol{E}_{f}(S)$ denote the collection of distinct elements in $E_{f}(S)$.\par
Let $g\in M(\mathbb{C})$. We say that two functions $f$ and $g$ share the set $S$ CM (resp. IM) if $E_{f}(S)=E_{g}(S)$ (resp. $\overline{E}_{f}(S)=\overline{E}_{g}(S)$).\par
F. Gross (\cite{G}) first studied the uniqueness problem of meromorphic functions that share distinct sets instead of values. Since then, the uniqueness theory of meromorphic functions under set sharing environment has become one of the important branch in the value distribution theory.\par
F. Gross (\cite{G}) proved that there exist three finite sets $S_{j}~(j=1,2,3)$  such that if two non-constant entire functions $f$ and $g$ share them, then $f\equiv g$. Later, in (\cite{G2}), he asked the following question:
\begin{ques}\label{q1}(\cite{G2})
Can one find two (or possible even one) finite set $S_{j}~(j=1,2)$ such that if two non-constant entire functions $f$ and $g$ share them, then $f\equiv g$?
\end{ques}
In 1982, F. Gross and C. C. Yang (\cite{GY}) first ensured the existence of such  set. They proved that if two non-constant entire functions $f$ and $g$ share the set $S=\{z\in \mathbb{C}:e^{z}+z=0\}$, then $f\equiv g$.\par
Moreover, this type of set was termed as a unique range set for entire functions. Later, similar definition for meromorphic functions was also introduced in literature.
\begin{defi}(\cite{YY})
Let $S\subset \mathbb{C}\cup\{\infty\}$; $f$ and $g$ be two non-constant meromorphic (resp. entire) functions. If $E_{f}(S)=E_{g}(S)$ implies $f\equiv g$, then $S$ is called a unique range set for meromorphic (resp. entire) functions or in brief URSM (resp. URSE).
\end{defi}
Here, we note that the set provided by Gross and Yang (\cite{GY}) was an infinite set. Thus after the introduction to the idea of unique range sets, many efforts were made to seek unique range sets with cardinalities as small as possible.\par
In 1994, H. X. Yi (\cite{H94}), settled the question of Gross by exhibitting a unique range set for entire functions with $15$ elements.\par Next year, P. Li and C. C. Yang (\cite{LY}) exhibited a unique range set for meromorphic (resp. entire) functions with $15$ (resp. $7$) elements. They considered the zero set of the following polynomial:
\begin{equation}\label{yi}
P(z)=z^{n}+az^{n-m}+b,
\end{equation}
where $a$ and $b$ are two non-zero constants such that $z^{n}+az^{n-m}+b=0$ has no multiple roots. Also, $m\geq 2~(resp. ~~1)$, $n > 2m+10~(resp. ~~2m+4)$ are integers with $n$ and $n-m$ having no common factors.\par
In 1996, H. X. Yi (\cite{H96}) further improved the result of Li and  Yang (\cite{LY}) and obtained a unique range set for meromorphic functions with $13$ elements.\par
In 2002, T. T. H. An (\cite{An}) exhibited another new class of unique range set for meromorphic functions with $13$ elements by considering the zero set of the following polynomial:
\begin{equation}\label{an}
P(z)=z^{n}+az^{n-m}+bz^{n-2m}+c,
\end{equation}
where it was assumed that $P(z)=0$ has no multiple roots, $a,b,c\in \mathbb{C}\setminus\{0\}$ such that $a^{2}\not=4b$. Also, $n$ and $2m$ are two positive integers such that $n$ and $2m$ have no common factors and $n > 8 + 4m$.\par
But till date the unique range set for meromorphic functions with $11$ elements is the smallest available unique range set for meromorphic functions obtained by G. Frank and M. Reinders (\cite{FR}). They studied the zero set of the following polynomial:
\begin{equation}\label{fr}
P(z)=\frac{(n-1)(n-2)}{2}z^{n}-n(n-2)z^{n-1}+\frac{n(n-1)}{2}z^{n-2}-c,
\end{equation}
where $n\geq 11$ and $c\not= 0,1$.\par
In 2007, T. C. Alzahary (\cite{A}) exhibited another new class of unique range set for meromorphic functions with $11$ elements by considering the zero set of the following polynomial:
\begin{equation}\label{al}
P(z)=az^{n}-n(n-1)z^{2}+2n(n-2)bz-(n-1)(n-2)b^{2},
\end{equation}
where $a$ and $b$ are two non-zero complex numbers satisfying $ab^{n-2}\not=1,2$ and $n\geq 11$. \par
Recently, another new class of unique range set for meromorphic functions with $11$ elements were exhibited in (\cite{BC1}) using the zero set of the following polynomial:
\begin{equation}\label{bcj}
P(z)=z^{n}-\frac{2n}{n-m}z^{n-m}+\frac{n}{n-2m}z^{n-2m}+c,
\end{equation}
where $c$ is any complex number satisfying $|c|\not=\frac{2m^2}{(n-m)(n-2m)}$ and $c\not=0,-\frac{1-\frac{2n}{n-m}+\frac{n}{n-2m}}{2}$ and $m\geq 1$, $n>\max\{ 2m+8,4m+1\}$.\par
\medbreak
Until now, the best results were given by H. Fujimoto (\cite{F1}) in 2000. He gave a generic unique range set for meromorphic function of at least $11$ elements when multiplicities are counted. To state his results, we need to explain some definitions.
\begin{defi}(\cite{F1, F2})
Let $P(z)$ be a non-constant monic polynomial. We call $P(z)$ a \enquote{uniqueness polynomial in broad sense} if  $P(f)\equiv P (g)$ implies $f \equiv g$ for any two non-constant meromorphic functions $f,~g$; while a \enquote{uniqueness polynomial} if  $P(f)\equiv cP (g)$ implies $f \equiv g$ for any two non-constant meromorphic functions $f,~g$ and non-zero constant $c$.
\end{defi}
For a discrete subset $S=\{a_{1}, a_{2},\ldots, a_{n}\}\subset \mathbb{C}$ ($a_{i}\not=a_{j}$), we consider the following polynomial
\begin{equation}\label{jkms1}
  P(z)=(z-a_{1})(z-a_{2})\ldots(z-a_{n})
\end{equation}
Assume that the derivative $P'(z)$ has $k$ distinct zeros $d_1, d_2, \ldots, d_k$ with multiplicities $q_1, q
_2,\ldots,q_{k}$ respectively. Under the assumption that
\begin{equation}\label{jkms2}
  P(d_{l_{s}})\not=P(d_{l_{t}})~~~~(1\leq l_{s}< l_{t}\leq k).
\end{equation}
H. Fujimoto (\cite{F1}) proved the following theorem:
\begin{theo}\label{1111}(\cite{F1})
 Let $P(z)$ be a \enquote{uniqueness polynomial} of the form (\ref{jkms1}) satisfying the condition (\ref{jkms2}). Moreover, either $k\geq3$ or $k=2$ and $\min\{q_{1},q_{2}\}\geq 2$.\par If $S$ is the set of zeros of $P(z)$, then $S$ is a unique range set for meromorphic (resp. entire) function whenever $n>2k+6$ (resp. $n>2k+2$).
\end{theo}
\begin{rem} We note that the Theorem \ref{1111} gives the best possible  generic unique range set for meromorphic function when $k=2$ (i.e., unique range set with $11$ elements).\par
Also, this type of unique range set (for $k=2$ case) was illustrated in (\cite{BC2}). The unique range set was the zero set of the following polynomial:
\begin{equation}\label{bc}
P(z)=Q(z)+c,
\end{equation}
where
$$Q(z)=\sum\limits_{i=0}^{m} \sum\limits_{j=0}^{n} \binom{m}{i}\binom{n}{j}\frac{(-1)^{i+j}}{n+m+1-i-j}z^{n+m+1-i-j}a^{j}b^{i},$$
 $a\not= b$, $b\not =0$, $c\not\in \{0,-Q(a),-Q(b),-\frac{Q(a)+Q(b)}{2}\}.$ Also, $m$, $n$ are two integers such that  $m+n> 9$, $\max\{m,n\}\geq 3$ and $\min\{m,n\}\geq 2$.
\end{rem}
\begin{rem}
  If we take $a = 0$ and $b = 1$ in the above example, then we get the unique range set for meromorphic functions with $11$ elements described in (\cite{BC}, \cite{BCA}).
\end{rem}
\section{Unique range sets with weight two}
Let $k$ be a non-negative integer or infinity. For $a\in\mathbb{C}\cup\{\infty\}$, we denote by $E_{k}(a;f)$, the set of all $a$-points of $f$, where an $a$-point of multiplicity $m$ is counted $m$ times if $m\leq k$ and $k+1$ times if $m>k$. \par
If for two meromorphic functions $f$ and $g$, we have $E_{k}(a;f)=E_{k}(a;g)$, then we say that $f$ and $g$ share the value $a$ with weight $k$ (\cite{L}). The IM and CM sharing respectively correspond to weight $0$ and $\infty$.\par
\medbreak
For $S\subset \mathbb{C}\cup\{\infty\}$, we define $E_{f}(S,k)$ as $$E_{f}(S,k)=\displaystyle\bigcup_{a\in S}E_{k}(a;f),$$
where $k$ is a non-negative integer or infinity. Clearly $E_{f}(S)=E_{f}(S,\infty)$.\par
Let $l\in \mathbb{N}\cup\{0\}\cup\{\infty\}$. A set $S\subset \mathbb{C} $ is called a $URSM_{l}$ (resp. $URSE_{l}$) if for any two non-constant meromorphic (resp. entire) functions $f$ and $g$, $E_{f}(S,l)=E_{g}(S,l)$ implies $f\equiv g$.\par
\medbreak
A recent development in the uniqueness theory of meromorphic functions is the introduction of the notion of weighted sharing instead of CM sharing. Also, it was observed that the cardinality of most of the existing unique range sets remain same if the sharing environment is relaxed from CM sharing to weighted sharing with weight two. In this direction, in 2012, A. Banerjee and I. Lahiri established the following remarkable result:
\begin{theo}\label{B}(\cite{BL}) Let $P(z)=a_{n}z^{n}+\sum\limits_{j=2}^{m}a_{j}z^{j}+a_{0}$ be a polynomial of degree $n$, where $n-m\geq 3$ and $a_{p}a_{m}\not=0$ for some positive integer $p$ with $2\leq p\leq m$ and $\gcd(p,3)=1$. Suppose further that $S=\{\alpha_{1},\alpha_{2},\ldots,\alpha_{n}\}$ be the set of all distinct zeros of $P(z)$. Let $k$ be the number of distinct zeros of the derivative $P'(z)$. If $n\geq 2k+7~(\text{resp.}~2k+3)$, then the following statements are equivalent:
\begin{enumerate}
\item [i)] $P$ is a \enquote{uniqueness polynomial} for meromorphic (resp. entire) function.
\item [ii)] $S$ is a $URSM_{2}$ (resp. $URSE_{2}$).
\item [iii)] $S$ is a URSM (resp. URSE).
\item [iv)] $P$ is a \enquote{uniqueness polynomial in broad sense} for meromorphic (resp. entire) function.
\end{enumerate}
\end{theo}
To prove the Theorem \ref{B}, the authors used the following lemma:
\begin{lem}\label{A}(\cite{BL}, Lemma 2.1) Let $P(z)=a_{n}z^{n}+\sum\limits_{j=2}^{m}a_{j}z^{j}+a_{0}$ be a polynomial of degree $n$, where $n-m\geq 3$ and $a_{p}a_{m}\not=0$ for some positive integer $p$ with $2\leq p\leq m$ and $\gcd(p,3)=1$. Suppose that $$\frac{1}{P(f)}=\frac{c_0}{P(g)}+c_{1},$$
where $f$ and $g$ are non-constant meromorphic functions and $c_0(\not=0)$, $c_{1}$ are constants. If $n\geq 6$, then $c_1=0$.
\end{lem}
It is noted that in Lemma \ref{A}, the condition $n-m\geq 3$ is essential. But, we observed that the \textbf{the condition \enquote{$n-m\geq 3$} is not necessary} in order to show  the equivalence between the statements (ii) and (iii) in \emph{Theorem \ref{B}}.
\medbreak
The motivation of this paper is to provide equivalency between a unique range set with counting multiplicity and unique range set with weight $2$, for
meromorphic functions. For this, we need to introduce the deficiency functions (\cite{YY}).\par
Let $a \in \mathbb{C}\cup \{\infty\}$, we set
$$\delta(a;f)=1-\limsup\limits_{r\rightarrow\infty}\frac{N(r,a;f)}{T(r,f)},$$
$$\Theta(a;f)=1-\limsup\limits_{r\rightarrow\infty}\frac{\ol{N}(r,a;f)}{T(r,f)}.$$
Now, we state the main result of this paper.
\begin{theo}\label{th1.1}
Let $P(z)$ be a polynomial of degree $n$ such that $P(z)=a_{0}(z-\alpha_{1})(z-\alpha_{2})\ldots(z-\alpha_{n})$; where $\alpha_{i}\not=\alpha_{j}$, $1\leq i,j\leq n$. Further suppose that $S=\{\alpha_{1},\alpha_{2},\ldots,\alpha_{n}\}$ be the set of all distinct zeros of $P(z)$. Let $k$ be the number of distinct zeros of the derivative $P'(z)$. Let $f$ and $g$ be two non-constant meromorphic functions such that $$\Theta(\infty;f)+\Theta(\infty;g)+\frac{1}{2}\min\{\delta(0,f),\delta(0,g)\}>\frac{2k+6-n}{2}.$$
Then the following two statements are equivalent:
\begin{enumerate}
\item [a)] If $E_{f}(S,2)=E_{g}(S,2)$, then $f\equiv g$.
\item [b)] If $E_{f}(S)=E_{g}(S)$, then $f\equiv g$.
\end{enumerate}
\end{theo}
\medbreak
The following corollary is the immediate consequence of Theorem \ref{th1.1}.
\begin{cor}\label{th1.2}
Let $P(z)$ be a polynomial of degree $n$ such that $P(z)=a_{0}(z-\alpha_{1})(z-\alpha_{2})\ldots(z-\alpha_{n})$; where $\alpha_{i}\not=\alpha_{j}$, $1\leq i,j\leq n$. Further suppose that $S=\{\alpha_{1},\alpha_{2},\ldots,\alpha_{n}\}$ be the set of all distinct zeros of $P(z)$. Let $k$ be the number of distinct zeros of the derivative $P'(z)$.  If $n\geq 2k+7~(resp.~2k+3)$, then the following two statements are equivalent:
\begin{enumerate}
\item [a)] $S$ is a $URSM_{2}$ (resp. $URSE_{2}$).
\item [b)] $S$ is a URSM (resp. URSE).
\end{enumerate}
\end{cor}
\medbreak
\begin{rem}
The polynomial described by the equation (\ref{yi})  satisfies the assumptions of the Corollary \ref{th1.2}. Here $k=m+1$, where $m\geq 2 ~(resp.~~1)$ is an integer for URSM (resp. URSE). Since the zero set of the polynomial gives URSM (resp. URSE) with $13$ (resp. $7$) elements, thus the zero set of the polynomial gives $URSM_{2}$ (resp. $URSE_{2}$) with $13$ (resp. $7$) elements.
\end{rem}
\begin{rem}
The polynomials described by the equations (\ref{fr}) and (\ref{bc})  satisfies the assumptions of the Corollary \ref{th1.2}. Here $k=2$. Since the zero set of the respective polynomials give URSM (resp. URSE) with $11$ (resp. $7$) elements, thus the zero set of the respective polynomials give $URSM_{2}$ (resp. $URSE_{2}$) with $11$ (resp. $7$) elements.
\end{rem}
\begin{rem}
The polynomial described by the equation (\ref{bcj})  satisfies the assumptions of the Corollary \ref{th1.2}. Here $k=m+1$, where $m\geq 1 (resp.~~1)$ is an integer. Since, for $m=1$, the zero set of the polynomial gives URSM (resp. URSE) with $11$ (resp. $7$) elements, thus the zero set of the polynomial gives $URSM_{2}$ (resp. $URSE_{2}$) with $11$ (resp. $7$) elements.
\end{rem}
\medbreak
Now we explain some definitions and notations which are used to prove the Theorem \ref{th1.1}.
\begin{defi} \label{d3}\cite{YY} Let $a\in\mathbb{C}\cup\{\infty\}$ and $m\in \mathbb{N}$.
\begin{enumerate}
\item [i)] We denote by $N(r,a;f\mid=1)$ the counting function of simple $a$-points of $f$.
\item [ii)] We denote by $N(r,a;f\mid\leq m)$ (resp. $N(r,a;f\mid\geq m)$ by the counting function of those $a$-points of $f$ whose multiplicities are not greater(resp. less) than $m$ where each $a$-point is counted according to its multiplicity.
\end{enumerate}
 Similarly, $\ol N(r,a;f \mid\leq m)$ and $\ol N(r,a;f \mid\geq m)$ are the reduced counting function of $N(r,a;f\mid\leq m)$ and $N(r,a;f\mid\geq m)$ respectively.
\end{defi}
\begin{defi}\label{d5}\cite{YY} Let $f$ and $g$ be two non-constant meromorphic functions such that $f$ and $g$ share $a$ IM. Let $z_{0}$ be an $a$-point of $f$ with multiplicity $p$, an $a$-point of $g$ with multiplicity $q$.
\begin{enumerate}
  \item [i)] We denote by $\ol N_{L}(r,a;f)$ the reduced counting function of those $a$-points of $f$ and $g$ where $p>q$.
  \item [ii)] We denote by $N^{1)}_{E}(r,a;f)$ the counting function of those $a$-points of $f$ and $g$ where $p=q=1$.
  \item [iii)] We denote by $\ol N^{(2}_{E}(r,a;f)$ the reduced counting function of those $a$-points of $f$ and $g$ where $p=q\geq 2$.
\end{enumerate}
In the same way we can define $\ol N_{L}(r,a;g)$, $N^{1)}_{E}(r,a;g)$, $\ol N^{(2}_{E}(r,a;g)$. When $f$ and $g$ share $a$ with weight $m$, $m\geq 1$ then $$N^{1)}_{E}(r,a;f)=N(r,a;f\mid=1).$$
\end{defi}
\begin{defi} \label{d7}\cite{YY} Let $f$, $g$ share a value $a$ IM. We denote by $\ol N_{*}(r,a;f,g)$ the reduced counting function of those $a$-points of $f$ whose multiplicities differ from the multiplicities of the corresponding $a$-points of $g$. Clearly $$\ol N_{*}(r,a;f,g) \equiv \ol N_{*}(r,a;g,f)~~\text{and}~~\ol N_{*}(r,a;f,g)=\ol N_{L}(r,a;f)+\ol N_{L}(r,a;g).$$
\end{defi}
\begin{proof}[\textbf{Proof of the Theorem \ref{th1.1} }]
 The case \textbf{$(a)\Rightarrow (b)$} is obvious. So, we only prove the case \textbf{$(b)\Rightarrow (a)$}.\par
 Let $f$ and $g$ be two non-constant meromorphic functions  share the set $$S=\{\alpha_{1},\alpha_{2},\ldots,\alpha_{n}\}$$ with weight $2$ and $$\Theta(\infty;f)+\Theta(\infty;g)+\frac{1}{2}\min\{\delta(0,f),\delta(0,g)\}>\frac{2k+6-n}{2}.$$ In this case, our claim is to show that $f\equiv g$. For that, we put $$P(z)=a_{0}(z-\alpha_{1})(z-\alpha_{2})\ldots(z-\alpha_{n}),$$
and
$$F(z):=\frac{1}{P(f(z))}~~\text{and}~~G(z):=\frac{1}{P(g(z))}.$$
Let $S(r)$ be any function $S(r):(0,\infty)\rightarrow\mathbb{R}$ satisfying $S(r)=o(T(r,F)+T(r,G))$ for $r\rightarrow\infty$ outside a set of finite Lebesgue Measure.\par
Let
$$H(z):=\frac{F''(z)}{F'(z)}-\frac{G''(z)}{G'(z)},$$
and this function $H$ was introduced by H. Fujimoto(\cite{F1}).\par
Now we consider two cases:\par
\textbf{Case-I} First we assume that $H\not\equiv 0$. It is given that
$$ 2k+6-2\Theta(\infty;f)-2\Theta(\infty;g)-\min\{\delta(0,f),\delta(0,g)\}+\epsilon<n,$$
where $\epsilon$ is a small positive number.\par
Since $H(z)$ can be expressed as
$$H(z)=\frac{G'(z)}{F'(z)}\left(\frac{F'(z)}{G'(z)}\right)',$$
so all poles of $H$ are simple. Also, \textbf{poles of $H$ may occur} at
\begin{enumerate}
 \item poles of $F$ and $G$.
  \item zeros of $F'$ and $G'$,
\end{enumerate}
Now, by simple calculations, one can show that \enquote{simple poles} of $F$ are the zeros of $H$. Thus
\bea \label{equn1.1} N(r,\infty;F|=1)=N(r,\infty;G|=1)\leq N(r,0;H). \eea
Now, using the lemma of logarithmic derivative and the first fundamental theorem, (\ref{equn1.1}) can be written as
\bea \label{equn1.2} N(r,\infty;F|=1)=N(r,\infty;G|=1)\leq N(r,\infty;H)+S(r). \eea
Let $\beta_{1},\beta_{2},\ldots,\beta_{k}$ be the $k$- distinct zeros of $P'(z)$. Since $F'(z)=-\frac{f'(z)P'(f(z))}{(P(f(z)))^{2}}$, $G'(z)=-\frac{g'(z)P'(g(z))}{(P(g(z)))^{2}}$ and $f$, $g$ share $S$ with weight $2$, by simple calculations, we can write
\bea \label{equn1.3} N(r,\infty;H)&\leq& \sum_{j=1}^{k}\left(\overline{N}(r,\beta_{j};f)+\overline{N}(r,\beta_{j};g)\right)+\ol{N}_{0}(r,0;f')+\ol{N}_{0}(r,0;g') \\
\nonumber &&+ \ol{N}(r,\infty;f)+\ol{N}(r,\infty;g)+\ol{N}_{\ast}(r,\infty;F,G),\eea
where $\ol{N}_{0}(r,0;f')$ denotes the reduced counting function of zeros of $f'$, which are not zeros of $\prod_{i=1}^{n}(f-\alpha_{i})\prod_{j=1}^{k}(f-\beta_{j})$, similarly, $\ol{N}_{0}(r,0;g')$ is defined.\par
Also, using Lemma 3 of (\cite{Y90}), we get
\bea\label{2.0} &&\ol{N}(r,\infty;F|\geq2)+\ol{N}_{0}(r,0,g')+\ol{N}_{\ast}(r,\infty;F,G)\\
\nonumber &\leq&\ol{N}(r,0;P(g)|\geq2)+\ol{N}_{0}(r,0,g')+\ol{N}(r,0;P(g)|\geq3)\\
\nonumber &\leq& N(r,0;g')\\
\nonumber &\leq& N(r,0;g)+\overline{N}(r,\infty;g)+S(r,g).
\eea
Put $T(r)=\max\{T(r,f),T(r,g)\}$ and $\delta(0)=\min\{\delta(0,f),\delta(0,g)\}$. Now, for any $\varepsilon (>0)$, using the second fundamental theorem and (\ref{equn1.2}), (\ref{equn1.3}) and (\ref{2.0}), we have
\bea \label{equn2.1} && (n+k-1)T(r,f)\\
\nonumber &\leq& \ol{N}(r,\infty;f)+\ol{N}(r,0;P(f))+\sum_{j=1}^{k}\overline{N}(r,\beta_{j};f)-\ol{N}_{0}(r,0;f')+S(r)\\
\nonumber &\leq& 2\ol{N}(r,\infty;f)+\ol{N}(r,\infty;g)+\sum_{j=1}^{k}\left(2\overline{N}(r,\beta_{j};f)+\overline{N}(r,\beta_{j};g)\right)\\
\nonumber &+&\ol{N}(r,\infty;F|\geq2)+\ol{N}_{0}(r,0;g')+\ol{N}_{\ast}(r,\infty;F,G)+S(r)\\
\nonumber &\leq& 2\ol{N}(r,\infty;f)+2\ol{N}(r,\infty;g)+2kT(r,f)+kT(r,g)+N(r,0;g)+S(r)\\
\nonumber &\leq& \left(3k+5-2\Theta(\infty;f)-2\Theta(\infty;g)-\delta(0)+\varepsilon\right)T(r)+S(r)\eea
Similarly,
\bea \label{equn2.2} && (n+k-1)T(r,g)\\
\nonumber &\leq& \left(3k+5-2\Theta(\infty;g)-2\Theta(\infty;f)-\delta(0)+\varepsilon\right)T(r)+S(r).
\eea
Thus comparing (\ref{equn2.1}) and (\ref{equn2.2}), we have
\bea \label{equn2.3} && (n+k-1)T(r)\\
\nonumber &\leq& \left(3k+5-2\Theta(\infty;g)-2\Theta(\infty;f)-\delta(0)+\varepsilon\right)T(r)+S(r),
\eea
which contradicts the assumption that $\Theta(\infty;f)+\Theta(\infty;g)+\frac{1}{2}\min\{\delta(0,f),\delta(0,g)\}>\frac{2k+6-n}{2}$. Hence $H\equiv 0$.\par
\textbf{Case-II}  Next we assume that $H\equiv 0$. Then by integration, we have
\beas \frac{1}{P(f(z))}&\equiv&\frac{c_0}{P(g(z))}+c_{1},\\
\text{i.e.,~~} \frac{1}{a_{0}(f-\alpha_{1})(f-\alpha_{2})\ldots(f-\alpha_{n})}&\equiv&\frac{c_0}{a_{0}(g-\alpha_{1})(g-\alpha_{2})\ldots(g-\alpha_{n})}+c_{1},\eeas
where $c_{0}$ is a non-zero complex constant. If $z_{0}$ be an $\alpha_{i}$ point of $f$ of multiplicity $m$, then it is a pole of $\frac{1}{P(f(z))}$ of order $m$, hence it is a pole of $\frac{1}{P(g(z))}$ of order $m$, i.e., $z_{0}$ is an $\alpha_{j}$ point of $g$ of order $m$ for some $j\in\{1,2,\ldots,n\}$.\par
Thus $f$ and $g$ share the set $S=\{\alpha_{1},\alpha_{2},\ldots,\alpha_{n}\}$ in counting multiplicity. Since $S$ is an URSM, so $f\equiv g$. This completes the proof.
\end{proof}
\section{Unique range sets with weak weight three}
Let $k$ be a non-negative integer or infinity. For $a\in\mathbb{C}\cup\{\infty\}$, we denote by $E_{k)}(a;f)$, the set of all $a$-points of $f$, whose multiplicities are not greater than $k$ and each such $a$-points are counted according to its multiplicity.\par
If for two meromorphic functions $f$ and $g$, we have $E_{k)}(a;f)=E_{k)}(a;g)$, then we say that $f$ and $g$ share the value $a$ with \enquote{weak-weight} $k$.\par
\medbreak
Let $S\subset \mathbb{C}\cup\{\infty\}$. We put $$E_{k)}(S, f)=\displaystyle\bigcup_{a\in S}E_{k)}(a;f),$$
where $k$ is a non-negative integer or infinity.\par
Let $l\in \mathbb{N}\cup\{0\}\cup\{\infty\}$. A set $S\subset \mathbb{C} $ is called a $URSM_{l)}$ (resp. $URSE_{l)}$) if for any two non-constant meromorphic (resp. entire) functions $f$ and $g$, $E_{l)}(S,f)=E_{l)}(S,g)$ implies $f\equiv g$.\par
In 2009, X. Bai, Q. Han and A. Chen (\cite{Bai}) proved the following \enquote{weak-weighted} sharing version of Fujimoto's Theorem:
\begin{theo}(\cite{Bai})
In addition to the hypothesis of Theorem \ref{1111}, further we suppose that $l\geq 3$  is a positive integer or $\infty$.\par
If $S$ is the set of zeros of $P(z)$ and $n>2k+6$ (resp. $n>2k+2)$, then $S$ is a $URSM_{l)}$ (resp. $URSE_{l)})$.
\end{theo}
Using the concept of weighted sharing and weak-weighted sharing, Banerjee and Lahiri (\cite{BL}) gave some equivalence between the two notions of unique range sets as follows:
\begin{theo}\label{tD}(\cite{BL}) Let $P(z)=a_{n}z^{n}+\sum\limits_{j=2}^{m}a_{j}z^{j}+a_{0}$ be a polynomial of degree $n$, where $n-m\geq 3$ and $a_{p}a_{m}\not=0$ for some positive integer $p$ with $2\leq p\leq m$ and $\gcd(p,3)=1$. Suppose further that $S=\{\alpha_{1},\alpha_{2},\ldots,\alpha_{n}\}$ be the set of all distinct zeros of $P(z)$. Let $k$ be the number of distinct zeros of the derivative $P'(z)$. If $n\geq 2k+7~(\text{resp.}~2k+3)$, then the following statements are equivalent:
\begin{enumerate}
\item [i)] $P$ is a \enquote{uniqueness polynomial} for meromorphic (resp. entire) function.
\item [ii)] $S$ is a $URSM_{3)}$ (resp. $URSE_{3)}$).
\item [iii)] $S$ is a URSM (resp. URSE).
\item [iv)] $P$ is a \enquote{uniqueness polynomial in broad sense} for meromorphic (resp. entire) function.
\end{enumerate}
\end{theo}
Here we also observed that to show the equivalence between the statements (ii) and (iii) in Theorem \ref{tD} one does not need the condition \enquote{$n-m\geq 3$}. Now, we state the result:
\begin{theo}\label{th2.1}
Let $P(z)$ be a polynomial of degree $n$ such that $P(z)=a_{0}(z-\alpha_{1})(z-\alpha_{2})\ldots(z-\alpha_{n})$; where $\alpha_{i}\not=\alpha_{j}$, $1\leq i,j\leq n$. Further suppose that $S=\{\alpha_{1},\alpha_{2},\ldots,\alpha_{n}\}$ be the set of all distinct zeros of $P(z)$. Let $k$ be the number of distinct zeros of the derivative $P'(z)$. Let $f$ and $g$ be two non-constant meromorphic functions such that $$\Theta(\infty;f)+\Theta(\infty;g)+\frac{1}{2}\min\{\delta(0,f),\delta(0,g)\}>\frac{2k+6-n}{2}.$$
Then the following two statements are equivalent:
\begin{enumerate}
\item [a)] If $E_{3)}(S,f)=E_{3)}(S,g)$, then $f\equiv g$.
\item [b)] If $E_{f}(S)=E_{g}(S)$, then $f\equiv g$.
\end{enumerate}
\end{theo}
The proof of this theorem is similar to the proof of Theorem \ref{th1.1}. So we omit it.
\begin{cor}\label{th2.2}
Let $P(z)$ be a polynomial of degree $n$ such that $P(z)=a_{0}(z-\alpha_{1})(z-\alpha_{2})\ldots(z-\alpha_{n})$; where $\alpha_{i}\not=\alpha_{j}$, $1\leq i,j\leq n$. Further suppose that $S=\{\alpha_{1},\alpha_{2},\ldots,\alpha_{n}\}$ be the set of all distinct zeros of $P(z)$. Let $k$ be the number of distinct zeros of the derivative $P'(z)$.  If $n\geq 2k+7~(resp.~2k+3)$, then the following two statements are equivalent:
\begin{enumerate}
\item [a)] $S$ is a $URSM_{3)}$ (resp. $URSE_{3)}$).
\item [b)] $S$ is a URSM (resp. URSE).
\end{enumerate}
\end{cor}
Combining Theorem \ref{th1.1} and Theorem \ref{th2.1}, the following conclusion is obvious:
\begin{cor}\label{th2.4}
Let $P(z)$ be a polynomial of degree $n$ such that $P(z)=a_{0}(z-\alpha_{1})(z-\alpha_{2})\ldots(z-\alpha_{n})$; where $\alpha_{i}\not=\alpha_{j}$, $1\leq i,j\leq n$. Further suppose that $S=\{\alpha_{1},\alpha_{2},\ldots,\alpha_{n}\}$ be the set of all distinct zeros of $P(z)$. Let $k$ be the number of distinct zeros of the derivative $P'(z)$.  Let $f$ and $g$ be two non-constant meromorphic functions such that $$\Theta(\infty;f)+\Theta(\infty;g)+\frac{1}{2}\min\{\delta(0,f),\delta(0,g)\}>\frac{2k+6-n}{2}.$$
Now the following statements are equivalent:
\begin{enumerate}
\item [a)] If $E_{f}(S)=E_{g}(S)$, then $f\equiv g$.
\item [b)] If $E_{f}(S,2)=E_{g}(S,2)$, then $f\equiv g$.
\item [c)] If $E_{3)}(S,f)=E_{3)}(S,g)$, then $f\equiv g$.
\end{enumerate}
\end{cor}
\section{Functions sharing two sets}
In 1994,  H. X. Yi (\cite{2Yi}) gave a positive answer to the Gross's question (Question \ref{q1}). Infact, he (\cite{2Yi}) gave  the existence of two finite sets $S_{1}$ (with $5$ elements) and $S_{2}$ (with one element) such that if any two non-constant entire functions $f$ and $g$ satisfying the condition $E(S_{j}, f) = E(S_{j}, g)$ for $j = 1, 2$, then $f\equiv g$.\par
Later in 1998, the same author (\cite{2Yi98}) proved that there exist two finite sets $S_{1}$ (with $3$ elements) and $S_{2}$ (with one element) such that any two non-constant entire functions $f$ and $g$ satisfying the condition $E(S_{j}, f) = E(S_{j}, g)$ for $j = 1, 2$ must be identical.\par
Moreover, H. X. Yi (\cite{2Yi98}), proved the following theorem:
\begin{theo}(\cite{2Yi98})
If $S_{1}$ and $S_{2}$ are two sets of  finite distinct complex numbers such that any two entire functions $f$ and $g$ satisfying $E_{f}(S_{j})= E_{g}(S_{j})$ for $j = 1,2$, must be identical, then $\max\{\sharp(S_1),\sharp(S_2)\} \geq 3$, where $\sharp(S)$ denotes the cardinality of the set $S$.
\end{theo}
Thus for the uniqueness of two entire functions when they share two sets, it is clear that the smallest cardinalities of $S_{1}$ and $S_{2}$ are $1$ and $3$ respectively.
\medbreak
Now, for the meromorphic functions, the Gross' Question should be the following:
\begin{ques}\label{q2}(\cite{2Yi})
Can one find two finite set $S_{j}~(j=1,2)$ such that if two non-constant meromorphic functions $f$ and $g$ share them, then $f\equiv g$?
\end{ques}
In 1994, H. X. Yi (\cite{2Yi}) completely answered the Question \ref{q2} by giving the existence of two finite sets $S_{1}$ (with $9$ elements) and $S_{2}$ (with $2$ elements) such that if any two non-constant meromorphic functions $f$ and $g$ satisfying the condition $E(S_{j}, f) = E(S_{j}, g)$ for $j = 1, 2$, then $f\equiv g$.\par
\medbreak
In this direction, in 2012, B. Yi and Y. H. Li (\cite{2YL}) provided a significant result. They proved that there exists two finite sets $S_{1}$ (with $5$ elements) and $S_{2}$ (with $2$ elements) such that if any two non-constant meromorphic functions $f$ and $g$ satisfying the condition $E(S_{j}, f) = E(S_{j}, g)$ for $j = 1, 2$, then $f\equiv g$.
\medbreak
The motivation of writting this section is to answer the Question \ref{q2} by giving the existence of two generic sets $S_{1}$ (with $?$ elements) and $S_{2}$ (with $2$ elements) for meromorphic function such that  if any two non-constant meromorphic functions $f$ and $g$ satisfying the condition $E(S_{j}, f) = E(S_{j}, g)$ for $j = 1, 2$, then $f\equiv g$.
\medbreak
Suppose
\begin{equation}\label{eb1}
P(z)=a_{0}(z-\alpha_{1})(z-\alpha_{2})\ldots(z-\alpha_{n}),
\end{equation}
 where $\alpha_{i}\not=\alpha_{j}$, $1\leq i,j\leq n$. Further suppose that
\begin{equation}\label{eb2}
P'(z)=b_{0}(z-\beta_{1})^{q_1}(z-\beta_{2})^{q_2}\ldots(z-\beta_{k})^{q_{k}}.
 \end{equation}
 Under the assumption (this property was introduced by H. Fujimoto (\cite{F1})) that
\begin{equation}\label{eb3}
  P(\beta_{l_{s}})\not=P(\beta_{l_{t}})~~~~(1\leq l_{s}< l_{t}\leq k).
\end{equation}
We state our main two theorems of this section:
\begin{theo}\label{thb1.1}
Let $P(z)$ be a \enquote{uniqueness polynomial} of the form (\ref{eb1}) satisfying the condition (\ref{eb3}). Further suppose that $S_{1}=\{\alpha_{1},\alpha_{2},\ldots,\alpha_{n}\}$ and $S_{2}=\{\beta_{1},\beta_{2},\ldots,\beta_{k}\}$.\par
If two non-constant meromorphic (resp. entire) functions $f$ and $g$ share the set $S_{1}$ with weight two and $S_{2}$ IM, $k\geq3$ and $n\geq k+7~(resp.~k+3)$, then $f\equiv g$.
\end{theo}
\begin{theo}\label{thbb1.1}
Let $P(z)$ be a \enquote{uniqueness polynomial} of the form (\ref{eb1}) satisfying the condition (\ref{eb3}). Further suppose that $S_{1}=\{\alpha_{1},\alpha_{2},\ldots,\alpha_{n}\}$ and $S_{2}=\{\beta_{1},\beta_{2},\ldots,\beta_{k}\}$.\par
Moreover, assume that $k \geq2$ and $P'(z)$ have no simple zeros. If two non-constant meromorphic (resp. entire) functions $f$ and $g$ share the set $S_{1}$ with weight $3$ and $S_{2}$ IM,  and $n\geq \max\{10-2k,5\}~(resp.~5)$, then $f\equiv g$.
\end{theo}
\begin{exm}
Let $$P(z)=\frac{(n-1)(n-2)}{2}z^{n}-n(n-2)z^{n-1}+\frac{n(n-1)}{2}z^{n-2}-c,$$ where $c\in\mathbb{C}\setminus\{0, \frac{1}{2}, 1\}$, $n\geq 6$.  In (\cite{bcs}), it was proved that $P(z)$ is a \enquote{uniqueness polynomial} satisfying the condition (\ref{eb3}). Suppose that  $S_{1}=\{z~:~\frac{(n-1)(n-2)}{2}z^{n}-n(n-2)z^{n-1}+\frac{n(n-1)}{2}z^{n-2}-c=0\}$ and $S_{2}=\{0,1\}$.\par Then by the Theorem \ref{thbb1.1}, if two non-constant meromorphic functions $f$ and $g$ share $S_{1}$ with weight $3$ and $S_2$ IM, then $f\equiv g$.\\
\end{exm}
To prove the above two theorems, we need the following lemma:
\begin{lem}\label{lemB}(\cite{F1}, Page- 1200)
Let $P(z)$ be a polynomial of the form (\ref{eb1}) satisfying the condition (\ref{eb3}). Suppose that $$\frac{1}{P(f)}=\frac{c_0}{P(g)}+c_{1},$$
where $f$ and $g$ are non-constant meromorphic functions and $c_0(\not=0)$, $c_{1}$ are constants. If $n\geq 5$ and $k\geq3$, or, if $n\geq 5$, $k=2$ and $P'(z)$ have no simple zeros, then $c_1=0$.
\end{lem}
\begin{proof}[\textbf{Proof of the Theorem \ref{thb1.1}}]  Since $f$ and $g$ share $S_{2}$ IM, so
$$\sum_{j=1}^{k}\overline{N}(r,\beta_{j};f)=\sum_{j=1}^{k}\overline{N}(r,\beta_{j};g).$$
Now, we put $$F(z):=\frac{1}{P(f(z))}~~\text{and}~~G(z):=\frac{1}{P(g(z))}.$$
Let $S(r)$ be any function $S(r):(0,\infty)\rightarrow\mathbb{R}$ satisfying $S(r)=o(T(r,F)+T(r,G))$ for $r\rightarrow\infty$ outside a set of finite Lebesgue Measure.\par
Let
$$H(z):=\frac{F''(z)}{F'(z)}-\frac{G''(z)}{G'(z)}.$$
Now, we consider two cases:\par
\textbf{Case-I} First we assume that $H\not\equiv 0$. Since $H(z)$ can be expressed as
$$H(z)=\frac{G'(z)}{F'(z)}\left(\frac{F'(z)}{G'(z)}\right)',$$
so all poles of $H$ are simple. Also, \textbf{poles of $H$ may occur} at
\begin{enumerate}
 \item poles of $F$ and $G$.
  \item zeros of $F'$ and $G'$,
\end{enumerate}
But using the Laurent series expansion of $H$, it is clear that \enquote{simple poles} of $F$ (hence, that of $G$) is a zero of $H$. Thus
\bea \label{equnb1.1} N(r,\infty;F|=1)=N(r,\infty;G|=1)\leq N(r,0;H) \eea
Using the lemma of logarithmic derivative and the first fundamental theorem, (\ref{equnb1.1}) can be written as
\bea \label{equnb1.2} N(r,\infty;F|=1)=N(r,\infty;G|=1)\leq N(r,\infty;H)+S(r) \eea
Since $F'(z)=-\frac{f'(z)P'(f(z))}{(P(f(z)))^{2}}$, $G'(z)=-\frac{g'(z)P'(g(z))}{(P(g(z)))^{2}}$ and $f$, $g$ share $(S_1,2)$ and $(S_2,0)$, by simple calculations, we can write
\bea \label{equnb1.3} N(r,\infty;H)&\leq& \sum_{j=1}^{k}\overline{N}(r,\beta_{j};f)+\ol{N}_{0}(r,0;f')+\ol{N}_{0}(r,0;g') \\
\nonumber &&+ \ol{N}(r,\infty;f)+\ol{N}(r,\infty;g)+\ol{N}_{\ast}(r,\infty;F,G)\eea
where $\ol{N}_{0}(r,0;f')$ denotes the reduced counting function of zeros of $f'$, which are not zeros of $\prod_{i=1}^{n}(f-\alpha_{i})\prod_{j=1}^{k}(f-\beta_{j})$, similarly, $\ol{N}_{0}(r,0;g')$ is defined.
Now, using the second fundamental theorem and (\ref{equnb1.2}), (\ref{equnb1.3}), we have
\bea \label{equnb1.4} && (n+k-1)\left(T(r,f)+T(r,g)\right)\\
\nonumber &\leq& \ol{N}(r,\infty;f)+\ol{N}(r,0;P(f))+\ol{N}(r,\infty;g)+\ol{N}(r,0;P(g))\\
\nonumber &+&\sum_{j=1}^{k}\left(\overline{N}(r,\beta_{j};f)+\overline{N}(r,\beta_{j};g)\right)-\ol{N}_{0}(r,0;f')-\ol{N}_{0}(r,0;g')+S(r)\\
\nonumber &\leq& 2\left(\ol{N}(r,\infty;f)+\ol{N}(r,\infty;g)\right)+\sum_{j=1}^{k}\left(2\overline{N}(r,\beta_{j};f)+\overline{N}(r,\beta_{j};g)\right)\\
\nonumber &+&\ol{N}(r,\infty;F|\geq2)+\ol{N}(r,\infty;G)+\ol{N}_{\ast}(r,\infty;F,G)+S(r).\eea
Noting that
\beas \ol{N}(r,\infty;F)-\frac{1}{2}N(r,\infty;F|=1)+\frac{1}{2}\ol{N}_{\ast}(r,\infty;F,G)\leq \frac{1}{2}N(r,\infty;F),\\
\ol{N}(r,\infty;G)-\frac{1}{2}N(r,\infty;G|=1)+\frac{1}{2}\ol{N}_{\ast}(r,\infty;F,G)\leq \frac{1}{2}N(r,\infty;G).
\eeas
Thus (\ref{equnb1.4}) can be written as
\bea \label{equnb1.5} && (n+k-1)\left(T(r,f)+T(r,g)\right)\\
\nonumber &\leq& 2\left(\ol{N}(r,\infty;f)+\ol{N}(r,\infty;g)\right)+(\frac{3}{2}k+\frac{n}{2})(T(r,f)+T(r,g))+S(r).\eea
which contradicts the assumption $n\geq k+7~~(\text{resp.}~ k+3)$. Thus $H\equiv 0$.\par
\textbf{Case-II} Next we assume that $H\equiv 0$. Then by integration, we have
\beas \frac{1}{P(f(z))}&\equiv&\frac{c_0}{P(g(z))}+c_{1},\eeas
where $c_{0}(\not= 0)$, $c_1$ are constants.\par
Here $n\geq5$ and $k\geq 3$. Thus, applying Lemma \ref{lemB} and noting the assumption that $P(z)$ is a \enquote{uniqueness polynomial}, we have $$f\equiv g.$$ This completes the proof.
\end{proof}
\medbreak
\begin{proof}[\textbf{Proof of the Theorem \ref{thbb1.1}}] Since $f$ and $g$ share $S_{2}$ IM, so
$$\sum_{j=1}^{k}\overline{N}(r,\beta_{j};f)=\sum_{j=1}^{k}\overline{N}(r,\beta_{j};g).$$
Now, we put $$F(z):=\frac{1}{P(f(z))}~~\text{and}~~G(z):=\frac{1}{P(g(z))}.$$
Let $S(r)$ be any function $S(r):(0,\infty)\rightarrow\mathbb{R}$ satisfying $S(r)=o(T(r,F)+T(r,G))$ for $r\rightarrow\infty$ outside a set of finite Lebesgue Measure.\par
Let
$$H(z):=\frac{F''(z)}{F'(z)}-\frac{G''(z)}{G'(z)}.$$
Now, we consider two cases:\par
\textbf{Case-I} First we assume that $H\not\equiv 0$. Now, proceeding as Case-I of Theorem \ref{thb1.1}, we have from (\ref{equnb1.4}) that
\bea \label{equnbb1.4} && (n+k-1)\left(T(r,f)+T(r,g)\right)\\
\nonumber &\leq& 2\left(\ol{N}(r,\infty;f)+\ol{N}(r,\infty;g)\right)+\sum_{j=1}^{k}\left(2\overline{N}(r,\beta_{j};f)+\overline{N}(r,\beta_{j};g)\right)\\
\nonumber &+&\ol{N}(r,\infty;F|\geq2)+\ol{N}(r,\infty;G)+\ol{N}_{\ast}(r,\infty;F,G)+S(r).\eea
Since $f$ and $g$ share the set $S_{1}$ with weight $3$, so
\beas &&\ol{N}(r,\infty;F)+\ol{N}(r,\infty;G)-N(r,\infty;F|=1)+\frac{5}{2}\ol{N}_{\ast}(r,\infty;F,G)\\
&&\leq \frac{1}{2}(N(r,\infty;F)+N(r,\infty;G)).
\eeas
Thus (\ref{equnbb1.4}) can be written as
\bea \label{equnbb1.5} && (n+k-1)\left(T(r,f)+T(r,g)\right)\\
\nonumber &\leq& 2\left(\ol{N}(r,\infty;f)+\ol{N}(r,\infty;g)\right)+3\sum_{j=1}^{k}\overline{N}(r,\beta_{j};f)\\
\nonumber &&+\frac{1}{2}(N(r,\infty;F)+N(r,\infty;G))-\frac{3}{2}\ol{N}_{\ast}(r,\infty;F,G)+S(r).\eea
Let
$$\varphi(z):=\frac{F'(z)}{F(z)}-\frac{G'(z)}{G(z)}.$$
Next we consider two cases:\\
\textbf{Subcase-I} Assume that $\varphi\not\equiv0$. Thus all poles of $\varphi$ are simple. Also, \textbf{poles of $\varphi$ may occur} at
\begin{enumerate}
 \item poles of $F$ and $G$,
  \item zeros of $F$ and $G$.
\end{enumerate}
Here we assume that $P'(z)$ have no simple zeros. Thus using the first fundamental theorem and the lemma of logarithmic derivative, we have
\beas 2\sum_{j=1}^{k}\overline{N}(r,\beta_{j};f)&\leq& N(r,0;\varphi)\\
&\leq& T(r,\varphi)+O(1),\\
&\leq& N(r,\infty;\varphi)+S(r,F)+S(r,G),\\
&\leq&\ol{N}(r,\infty;f)+\ol{N}(r,\infty;g)+\ol{N}_{\ast}(r,\infty;F,G)+S(r).
\eeas
Thus (\ref{equnbb1.5}) can be written as
\bea \label{equnbb1.6} && (\frac{n}{2}+k-1)\left(T(r,f)+T(r,g)\right)\\
\nonumber &\leq&\frac{7}{2}\left(\ol{N}(r,\infty;f)+\ol{N}(r,\infty;g)\right)+S(r),\eea
which is impossible if $f$ and $g$ both are entire functions. Also, this is impossible if $f$ and $g$ both are meromorphic  functions and $n\geq 10-2k$. Thus $H\equiv 0$.\\
\textbf{Subcase-II} Next we assume that $\varphi\equiv 0$. Then on integration, we have $$F\equiv A G,$$ where $A$ is a non-zero constant, which is impossible as $H\not\equiv0$.\\
\textbf{Case-II} Next we assume that $H\equiv 0$.  Then by integration, we have
\beas \frac{1}{P(f(z))}&\equiv&\frac{c_0}{P(g(z))}+c_{1},\eeas
where $c_{0}(\not= 0)$, $c_1$ are constants.\par
Here $n\geq5$, $k\geq 2$ and $P'(z)$ have no simple zeros. Thus, applying Lemma \ref{lemB} and noting the assumption that $P(z)$ is a \enquote{uniqueness polynomial}, we have $$f\equiv g.$$ This completes the proof.
\end{proof}
\section{Some Observations}
The natural query would be whether there exists different classes of unique range sets, but in this direction the number of results are not sufficient. Recently, V. H. An and P. N. Hoa (\cite{AH}) exhibited a new class of  unique range set for meromorphic functions. The unique range set is the zero set of the following polynomial:
\begin{equation}\label{Ah}
P(z)=z^{n}+(az+b)^{n}+c,
\end{equation}
where $n\geq 25$ is an integer, $a,b,c\in \mathbb{C}\setminus\{0\}$ with $c\not=\frac{b^d}{a^d}$, $a^{2d}\not=1$, $c\not=a^{d}b^{d}$, $c\not=\frac{(-1)^{d}b^{d}}{a^{2d}}$, $c\not=(-1)^{d}b^{d}$. Also, it was assumed that $P(z)$ has only simple zeros.\par
This URSM has $25$ elements but in literature there exist URSM with $11$ elements. Also, it was proved that any URSM (resp.URSE) must contain at least six (resp. five) [see, Theorem 10.59 (resp. Theorem 10.72), (\cite{YY})] elements. So, the challenging work is to exhibit URSM (resp. URSE) with elements $\leq 11(resp.~~ 7)$.
\medbreak
Now, we discuss a method of P. Li and  C. C. Yang (page 448, (\cite{LY})):\par
Let $S=\{\alpha_{1},\alpha_{2},\ldots,\alpha_{n}\}$ be a set with finite distinct elements of $\mathbb{C}$. Also, let $\alpha(\not=0)$ and $\beta$ be two complex constants. If $S$ is a unique range set, then the set $T=\{\alpha\alpha_{1}+\beta,\alpha\alpha_{2}+\beta,\ldots,\alpha\alpha_{n}+\beta\}$ is also a unique range set. \par
If $f$ and $g$ are two meromorphic functions sharing $T$ CM, then
\beas && (f-(\alpha\alpha_{1}+\beta))(f-(\alpha\alpha_{2}+\beta))\ldots(f-(\alpha\alpha_{n}+\beta))\\
&&=h(g-(\alpha\alpha_{1}+\beta))(g-(\alpha\alpha_{2}+\beta))\ldots(g-(\alpha\alpha_{n}+\beta)),\eeas
where $h$ is a meromorphic function whose zeros come from the poles of $g$ and the poles come from the poles of $f$. Thus\par
\beas && \left(\frac{f-\beta}{\alpha}-\alpha_{1}\right)\left(\frac{f-\beta}{\alpha}-\alpha_{2}\right)\ldots\left(\frac{f-\beta}{\alpha}-\alpha_{n}\right)\\
&&=h\left(\frac{g-\beta}{\alpha}-\alpha_{1}\right)\left(\frac{g-\beta}{\alpha}-\alpha_{2}\right)\ldots\left(\frac{g-\beta}{\alpha}-\alpha_{n}\right).\eeas
Thus $\frac{f-\beta}{\alpha}$ and $\frac{g-\beta}{\alpha}$ share $S$ CM. So, $\frac{f-\beta}{\alpha}\equiv \frac{g-\beta}{\alpha}$, i.e., $f\equiv g$. So, $T$ is a URSM.
\begin{rem} Since the examples of unique range sets  are few in numbers, thus this method ((page 448, \cite{LY}) helps us to construct new class of unique range sets. For examples,
\begin{enumerate}
\item [i)] The zero set of the following polynomial gives a new class of URSM (resp. URSE) with $13~(resp.~~7)$ elements:
$$P(z)=(z-\beta)^{n}+a(z-\beta)^{n-m}+b,$$
where $\beta\in \mathbb{C}$, $a$ and $b$ are two non-zero constants such that $z^{n}+az^{n-m}+b=0$ has no multiple root. Also, $m\geq 2~(resp.~~1)$, $n > 2m+8~(resp.~~2m+4)$ are integers with $n$ and $n-m$ having no common factors.
\item [ii)] The zero set of the following polynomial gives a new class of URSM with $11$ elements:
$$P(z)=\frac{(n-1)(n-2)}{2}(z-\beta)^{n}-n(n-2)(z-\beta)^{n-1}+\frac{n(n-1)}{2}(z-\beta)^{n-2}-c,$$
where $n\geq 11$ and $c\not= 0,1$, $\beta\in \mathbb{C}$.
\end{enumerate}
\end{rem}
$$~~~~~~~~~~~~~~~~~~~~$$
\medbreak
We have seen from the equation (\ref{yi}) that the zero set of the polynomial  $$P(z)=z^{n}+z^{n-1}+1$$ gives a URSE with $n(\geq 7)$ elements. Now, $$z^{n}P\left(\frac{1}{z}\right)=z^{n}+z+1.$$
Also, the zero set of the polynomial  $z^{n}P(\frac{1}{z})=z^{n}+z+1$ gives a URSE with $n(\geq 7)$ elements (Theorem 10.57, (\cite{YY})).\par
Again, the equation (\ref{fr}) gives that the zero set of the polynomial $$P(z)=\frac{(n-1)(n-2)}{2}z^{n}-n(n-2)z^{n-1}+\frac{n(n-1)}{2}z^{n-2}-\frac{c}{2},$$ where $c\not= 0,1,2$ is a URSM with $n(\geq 11)$ elements. Here,
$$z^{n}P\left(\frac{1}{z}\right)=-\frac{1}{2}\left(cz^{n}-n(n-1)z^{2}+2n(n-2)z-(n-1)(n-2)\right).$$
From the equation (\ref{al}), the zero set of the polynomial  $z^{n}P(\frac{1}{z})$ also gives a URSM with $n(\geq 11)$ elements.\par
Thus the following question is obvious:
\begin{ques} Let $P(z)$  be a non-constant polynomial having simple zeros. What are the characterizations of the polynomial $P(z)$ such that if the zero set of the polynomial $P(z)$ forms a unique range set, then the zero set of polynomial $z^{n}P(\frac{1}{z})$ must form a unique range set?
\end{ques}
\begin{center} {\bf Acknowledgement} \end{center}
The authors are  grateful to the anonymous referees for their valuable suggestions which considerably improved the presentation of the paper.\par
The research work of the first and the second authors are  supported by the Department of Higher Education, Science and Technology \text{\&} Biotechnology, Govt. of West Bengal under the sanction order no. 216(sanc) /ST/P/S\text{\&}T/16G-14/2018 dated 19/02/2019.\par
The third and fourth authors are thankful to the Council of Scientific and Industrial Research, HRDG, India for granting Junior Research Fellowship (File No.: 09/106(0179)/2018-EMR-I \text{\&}  08/525(0003)/2019-EMR-I respectively) during the tenure of which this work was done.


\begin{thebibliography}{99}
\bibitem{A} T. C. Alzahary, Meromorphic functions with weighted sharing of one set, Kyungpook Math. J., 47 (2007), 57-68.
\bibitem{An} T. T. H. An, Unique range sets for meromorphic functions constructed without an injectivity hypothesis, Taiwanese J. Math. 15(2) (2011), 697-709.
\bibitem{AH} V. H. An and P. N. Hoa, A new class of unique range sets for meromorphic functions, Annales Univ. Sci. Budapest., Sect. Comp., 47 (2018) 109-116.
\bibitem{Bai} X. Bai, Q. Han and A. Chen, On a result of H. Fujimoto, J. Math. Kyoto University,  49(3) (2009), 631-643.
\bibitem{BC} A. Banerjee, A new class of strong uniqueness polynomial satisfying Fujimoto’s conditions, Ann. Acad. Sci. Fenn. Math. 40 (2015) 465–474.
\bibitem{BCA}  A. Banerjee and B. Chakraborty, A new type of unique range set with deficient values, Afr. Mat., 26 (2015), 1561-1572
\bibitem{BC1} A. Banerjee and B. Chakraborty, Further results on the uniqueness of meromorphic functions  and their derivative counterpart sharing one or two sets, Jordan J. Math. Stat., 9 (2), (2016), 117-139.
\bibitem{BC2} A. Banerjee and B. Chakraborty, On some sufficient conditions of strong uniqueness polynomials, Adv. Pure Appl. Math., 8 (1) (2017), 1-13.
\bibitem{bcs} A. Banerjee, B. Chakraborty and S. Mallick, Further Investigations on Fujimoto Type Strong Uniqueness Polynomials, Filomat 31, No. 16 (2017), 5203-5216.
\bibitem{BL} A. Banerjee and I. Lahiri, A uniqueness polynomial generating a unique range set and vise versa, Comput. Methods Funct. Theory, 12(2) (2012), 527-539.
\bibitem{FR} G. Frank and M. Reinders, A unique range set for meromorphic functions with $11$ elements, Complex Var. Theory Appl., 37(1) (1998), 185-193.
\bibitem{F1} H. Fujimoto, On uniqueness of meromorphic functions sharing finite sets, Amer. J. Math., 122 (2000), 1175-1203.
\bibitem{F2} H. Fujimoto, On uniqueness polynomials for meromorphic functions, Nagoya Math. J., 170 (2003), 33-46.
\bibitem{G} F. Gross, On the distribution of values of meromorphic functionas, Trans. Amer. Math. Soc., 131 (1968), 199-214.
\bibitem{G2} F. Gross, Factorization of meromorphic functions and some open problems, Complex Analysis, Lecture Notes in Math., 599 (Proc. Conf. Univ. Kentucky, Lexington, Ky, 1976), Berlin, Heidelberg, New York: Spring-Verlag, (1977), 51-67.
\bibitem{GY} F. Gross and C. C. Yang, On preimage and range sets of meromorphic functions, Proc. Japan Acad., 58 (1982), 17-20.
\bibitem{H} W. K. Hayman, Meromorphic Functions, The Clarendon Press, Oxford, 1964.
\bibitem{L} I. Lahiri, Weighted sharing and uniqueness of meromorphic functions, Nagoya Math. J., 161 (2001), 193-206.
\bibitem{LY} P. Li and C. C. Yang, Some further results on the unique range sets of meromorphic functions, Kodai Math. J., 18 (1995), 437-450.
\bibitem{YY} C. C. Yang and H. X. Yi, Uniqueness Theory of Meromorphic Functions, Mathematics and its Applications, 557, Kluwer Academic Publishers Group, Dordrecht, 2003.
\bibitem{Y90} H. X. Yi, Uniqueness of meromorphic functions and a question of C. C. Yang,  Complex Var. Theory Appl., 14(1-4) (1990), 169-176.
\bibitem{H94} H. X. Yi, On a problem of Gross, Sci. China, Ser. A, 24 (1994), 1137-1144.
\bibitem{2Yi} H. X. Yi, Uniqueness of meromorphic functions and question of Gross, Science in China (Series A), 37 (1994), 802-813.
\bibitem{H96} H. X. Yi, Unicity theorems for meromorphic or entire functions III, Bull. Austral. Math. Soc.,53 (1996), 71-82.
\bibitem{2Yi98} H. X. Yi, On a question of Gross concerning uniqueness of entire functions, Bull. Austral. Math. Soc., 57 (1998), 343-349.
\bibitem{2YL} B. Yi and Y. H. Li, The uniqueness of meromorphic functions that share two sets with CM, Acta Math. Sin., Chin. Ser., 55 (2012), 363-368.
\end{thebibliography}
\end{document}